\documentclass[12pt]{amsart}
\usepackage{geometry}
\usepackage{hyperref}
\usepackage[all]{xy}
\usepackage{tikz}
\usetikzlibrary{positioning}

\newtheorem{thm}{Theorem}[section]
\newtheorem{lem}[thm]{Lemma}
\newtheorem{cor}[thm]{Corollary}
\newtheorem{prop}[thm]{Proposition}
\newtheorem*{thm*}{Theorem}

\theoremstyle{definition}

\newtheorem{example}[thm]{Example}

\newtheorem{que}[thm]{Question}

\theoremstyle{remark}

\numberwithin{equation}{section}

\newcommand{\bP}{\mathbb{P}}

\newcommand{\bQ}{\mathbb{Q}}

\newcommand{\cO}{\mathcal{O}}

\newcommand{\cI}{\mathcal{I}}

\begin{document}

\title{Smooth rational curves on singular rational surfaces}

\author{Ziquan Zhuang}

\address{Department of Mathematics, Princeton University, Princeton, NJ, 08544-1000.}

\email{zzhuang@math.princeton.edu}

\begin{abstract}
We classify all complex surfaces with quotient singularities that do not contain any smooth rational curves, under the assumption that the canonical divisor of the surface is not pseudo-effective. As a corollary we show that if $X$ is a log del Pezzo surface such that for every closed point $p\in X$, there is a smooth curve (locally analytically) passing through $p$, then $X$ contains at least one smooth rational curve.
\end{abstract}

\maketitle

\section{Introduction}

Let $X$ be a projective rationally connected variety defined over
$\mathbb{C}$. When $X$ is smooth, it is well known that there are
many smooth rational curves on $X$: if $\dim X=2$ then $X$ is
isomorphic to a blowup of either $\mathbb{P}^{2}$ or a ruled
surface $\mathbb{F}_{e}$; if $\dim X\ge3$, any two points on $X$
can be connected by a very free rational curve, i.e. image of $f:\mathbb{P}^{1}\rightarrow X$
such that $f^{*}T_{X}$ is ample, and a general deformation of $f$
is a smooth rational curve on $X$ (for the definition of rationally
connected variety and the above mentioned properties, see \cite{RationalCurve}).
It is then natural to ask about the existence of smooth rational curves
on $X$ when $X$ is singular. In this paper, we study this problem
on rational surfaces.

There are some possible obstructions to the existence of smooth rational
curves. It could happen that there is no smooth curve germ passing
through the singular points of $X$ (e.g. when $X$ has $E_{8}$ singularity)
while the smooth locus of $X$ contains no rational curves at all
(this could be the case when the smooth locus is of log Calabi-Yau
or log general type), and then we won't be able to find any smooth
rational curves on $X$. Hence to produce smooth rational curves on
$X$, we will need some control on the singularities of $X$ and the
``negativity'' of its smooth locus. We will show that these restrictions
are also sufficient, in particular, we will prove the following theorem,
which is one of the main results of this paper:

\begin{thm*}
Let $X$ be a surface with only quotient singularities. Assume that
\begin{enumerate}
\item $K_{X}$ is not pseudo-effective;
\item For every closed point $p\in X$, there is a smooth curve (locally
analytically) passing through $p$.
\end{enumerate}
Then $X$ contains at least one smooth rational curve.
\end{thm*}

In fact, we will prove something stronger. By studying various adjoint
linear systems on rational surfaces, we show that condition (1) above
combined with nonexistence of smooth rational curves has strong implication
on the divisor class group of $X$ (Proposition \ref{prop:criterion for no P1}),
which allows us to classify all surfaces with quotient singularities
that satisfy condition (1) above but do not contain smooth rational
curves (Theorem \ref{thm:Classification}). It turns out that all such
surfaces have an $E_{8}$ singularity, which is the only surface quotient
singularity that does not admit a smooth curve germ.

This paper is organized as follows. In section 2 we study the existence
of smooth rational curves on rational surfaces with quotient singularities
whose anticanonical divisor is pseudo-effective but not numerically
trivial and give the proof of the main result. In section 3 we study
some examples and propose a few questions. In particular we construct some rational surfaces with
quotient singularity and numerically trivial canonical divisor that
contain no smooth rational curves.

\smallskip

\paragraph*{\bf Conventions}

We work over the field $\mathbb{C}$ of complex numbers. Unless mentioned
otherwise, all varieties in this paper are assumed to be proper and all surfaces normal. 
A surface $X$ is called log del Pezzo if there is a $\mathbb{Q}$-divisor $D$
on $X$ such that $(X,D)$ is klt and $-(K_{X}+D)$ is ample.

\smallskip

\paragraph*{\bf Acknowledgement}
The author would like to thank his advisor J\'anos Koll\'ar for constant
support and lots of inspiring conversations. He also wishes to thank Qile
Chen, Ilya Karzhemanov, Brian Lehmann, Yuchen Liu, Chenyang Xu and Yi
Zhu for many helpful discussions. Finally he is grateful to the anonymous referee for careful reading of his manuscript and for the numerous constructive comments.

\section{Proof of main theorem}

In this section, we will classify all surfaces with quotient
singularities containing no smooth rational curves, under the assumption
that the anticanonical divisor is pseudo-effective but not numerically
trivial. As a corollary, we will see that if $X$ is a log del Pezzo
surface that has no $E_{8}$ singularity (as $E_{8}$ is the only
surface quotient singularity whose fundamental cycle contains no reduced
component, by \cite{SmoothArc} this is equivalent to saying that
for every point $p\in X$, there is a smooth curve germ passing through
$p$), then $X$ contains at least one smooth rational curve.

We start by introducing a few results on adjoint linear systems that we frequently use to identify smooth rational curves on a surface.

\begin{lem}
\label{lem:K_X+D}Let $X$ be a smooth rational surface and $D$ a
reduced divisor on $X$, then $|K_{X}+D|=\emptyset$ if and only if every connected
component of $D$ is a rational tree (i.e. every irreducible component
of $D$ is a smooth rational curve and the dual graph of $D$ is a
disjoint union of trees).
\end{lem}

\begin{proof}
We have an exact sequence $0\rightarrow\omega_{X}\rightarrow\omega_{X}(D)\rightarrow\omega_{D}\rightarrow0$
which induces a long exact sequence
\[
\cdots\rightarrow H^{0}(X,\omega_{X})\rightarrow H^{0}(X,\omega_{X}(D))\rightarrow H^{0}(D,\omega_{D})\rightarrow H^{1}(X,\omega_{X})\rightarrow\cdots
\]
Since $X$ is a smooth rational surface, $H^{0}(X,\omega_{X})=H^{1}(X,\omega_{X})=0$,
hence $H^{0}(X,\omega_{X}(D))=0$ if and only if $H^{0}(D,\omega_{D})=0$.
We now show that the latter condition holds if and only if every connected
component of $D$ is a rational tree. By doing this, we may assume $D$ is connected. Since $D$ is reduced, $H^{0}(D,\omega_{D})=0$ is
equivalent to $p_{a}(D)=0$. Let $D_{i}(i=1,\cdots,k)$ be the irreducible
components of $D$, we have $0=p_{a}(D)=\sum_{i=1}^{k}p_{a}(D_{i})+e-v+1$
where $e$, $v$ are the number of edges and vertices in the dual
graph of $D$. Since each $p_{a}(D_{i})\ge0$ and $e-v+1\ge0$, we
have equality everywhere, hence the lemma follows.
\end{proof}

We also need an analogous result when $X$ is not smooth.

\begin{lem}
Let $X$ be a projective normal Cohen-Macaulay variety of dimension
at least 2 and $D$ a Weil divisor on $X$, then we have an exact
sequence
\begin{equation}
0\rightarrow\omega_{X}\rightarrow\mathcal{O}_{X}(K_{X}+D)\rightarrow\omega_{D}\rightarrow0\label{eq:residue}
\end{equation}
where $\omega_{X}$, $\omega_{D}$ are the dualizing sheaf of $X$
and $D$, and $K_{X}$ is the canonical divisor of $X$.
\end{lem}

\begin{proof}
See \cite[4.1]{KK13}
\end{proof}

\begin{cor}
\label{cor:when a curve is P1}Let $X$ be a rational surface with
only rational singularities, $D$ an integral curve on $X$, then
$D$ is a smooth rational curve if and only if $|K_{X}+D|=\emptyset$.
\end{cor}

\begin{proof}
Since $X$ is a normal surface, it is CM, so we can apply the previous
lemma to get the exact sequence \eqref{eq:residue}, which induces
the long exact sequence
\[
H^{0}(X,\omega_{X})\rightarrow H^{0}(X,\mathcal{O}_{X}(K_{X}+D))\rightarrow H^{0}(D,\omega_{D})\rightarrow H^{1}(X,\omega_{X})\rightarrow\cdots
\]
As $X$ is a rational surface with only rational singularities, we
have $H^{0}(X,\omega_{X})=H^{1}(X,\omega_{X})=0$, hence $D$ is a
smooth rational curve iff $H^{0}(D,\omega_{D})=0$ iff $|K_{X}+D|=\emptyset$.
\end{proof}

One may notice that the above lemmas only apply to rational surfaces while our main theorem is stated for arbitrary surfaces. This is only a minor issue, as illustrated by the following lemma.

\begin{lem} \label{lem:reduction}
Let $X$ be a surface. Assume that $X$ does not contain any smooth rational curves. Then either $K_X$ is nef or $-K_X$ is numerically ample and $\rho(X)=1$. 
\end{lem}

Here since $-K_X$ is in general not $\bQ$-Cartier, its nefness or numerical ampleness is understood in the sense of \cite{sakaiams}. In particular, if we further assume $X$ has rational singularities (which implies $X$ is $\bQ$-factorial) and $K_X$ is not pseudo-effective (as we do in our main theorem), then $-K_X$ is ample and $X$ is a rational surface of Picard number one by \cite[Lemma 3.1]{numdP}.

\begin{proof}
First suppose $X$ is not relatively minimal. By \cite[Theorem 1.4]{sakaiams}, we may run the $K_X$-MMP on $X$. Let $f:X\rightarrow Y$ be the first step in the MMP. Since $-K_X$ is $f$-ample, by \cite[Theorem 6.3]{sakaiduke} we have $R^1f_*\cO_X=0$. Let $C\subseteq X$ be an irreducible curve contracted by $f$ and $\cI_C$ its ideal sheaf. Since the fibers of $f$ has dimension $\le 1$ we have $R^2f_*\cI_C=0$ by the theorem of formal functions. It then follows from the long exact sequence associated to $0\rightarrow\cI_C\rightarrow\cO_X\rightarrow\cO_C\rightarrow0$ that $H^1(C,\cO_C)=R^1f_*\cO_C=0$, hence $C$ is a smooth rational curve on $X$, contrary to our assumption.

We may therefore assume that $X$ is relatively minimal. If $K_X$ is not nef then by \cite[Theorem 3.2]{sakaiams}, either $-K_X$ is numerically ample and $\rho(X)=1$ or $X$ admits a fibration $g:X\rightarrow B$ whose general fiber is $\bP^1$. However, the latter case cannot occur since $X$ does not contain smooth rational curves. This proves the lemma. 
\end{proof}

Now we come to a useful criterion for whether a surface
contains at least one smooth rational curve.

\begin{prop}
\label{prop:criterion for no P1}Let $X$ be a surface with
only rational singularities. Assume $K_{X}$ is not pseudo-effective, then the following are equivalent:
\begin{enumerate}
\item $X$ does not contain any smooth rational curves;
\item The class group $\mathrm{Cl}(X)$ is infinite cyclic and is generated
by some effective divisor $D$ linearly equivalent to $-K_{X}$.
\end{enumerate}
\end{prop}

\begin{proof}
First assume (2) holds. By \cite[Lemma 3.1]{numdP}, $X$ is necessarily a rational surface. If $X$ contains a smooth rational curve $C$,
then by Corollary \ref{cor:when a curve is P1}, $|K_{X}+C|=\emptyset$,
but by (2), we may write $C\sim kD$ for some integer $k\ge1$, and
$K_{X}+C\sim(k-1)D$ is effective, a contradiction, so (1) follows.

Now assume (1) holds. By Lemma \ref{lem:reduction} and its subsequent remark, $X$ is a rational surface with ample anti-canonical divisor. Let $H$ be an ample divisor on $X$ and assume
there exists some effective divisor $C$ on $X$ that is not an integral
multiple of $-K_{X}$ in $\mathrm{Cl}(X)$. Among such divisors we
may choose $C$ so that $(H.C)$ is minimal. Clearly $C$ is integral,
and by (1) it is not a smooth rational curve, hence by Corollary \ref{cor:when a curve is P1},
$K_{X}+C$ is effective. Since $-K_{X}$ is ample, we have
$(K_{X}+C.H)<(C.H)$, so by our choice of $C$, $K_{X}+C$ is an integral
multiple of $K_{X}$, hence so is $C$, a contradiction. It follows
that every effective divisor on $X$ is linearly equivalent to a multiple
of $-K_{X}$. Since $\mathrm{Cl}(X)$ is generated by the class of
effective divisors, we see that it is infinite cyclic and generated
by $-K_{X}$. Now let $m$ be the smallest positive integer such that
$-mK_{X}$ is effective. Write $-mK_{X}\sim\sum a_{i}D_{i}$ where
$a_{i}>0$ and $D_{i}$ is integral. As $m$ is minimal and each $D_{i}$
is also a multiple of $-K_{X}$, we have indeed $-mK_{X}\sim D$ an
integral curve. $D$ is not smooth rational by (1), hence again by
Corollary \ref{cor:when a curve is P1}, $K_{X}+D$ is effective,
but $K_{X}+D\sim-(m-1)K_{X}$, so by the minimality of $m$ we have
$m=1$, and thus all the assertions in (2) are proved.
\end{proof}

From now on, $X$ will always be a normal surface that satisfies
the assumptions and the equivalent conditions (1)(2) of Proposition
\ref{prop:criterion for no P1}. In particular, $X$ is rational, $\mathbb{Q}$-factorial and has Picard number one, $-K_{X}$
is ample and $\mathrm{Pic}(X)\cong\mathbb{Z}$ is generated by $-rK_{X}$
where $r$ is smallest positive integer such that $rK_{X}$ is Cartier
(i.e. the index of $X$). We further assume that $X$ has at worst
quotient singularities (or equivalently, klt singularities, as we
are in the surface case). Let $X^{0}$ be the smooth locus of $X$,
$\pi:Y\rightarrow X$ the minimal resolution and $E\subset Y$ the
reduced exceptional locus.

\begin{lem}
Notation as above. Then we have an exact sequence
\[
0\rightarrow\mathrm{Cl}(X)/\mathrm{Pic}(X)\rightarrow H^{2}(E,\mathbb{Z})/H_{2}(E,\mathbb{Z})\rightarrow H_{1}(X^{0},\mathbb{Z})\rightarrow0
\]
and an isomorphism $H^{2}(E,\mathbb{Z})/H_{2}(E,\mathbb{Z})\cong\mathbb{Z}/r\mathbb{Z}$. 
\end{lem}

Here we identify $H_{2}(E,\mathbb{Z})$ as a subgroup of $H^{2}(E,\mathbb{Z})$
by the composition $H_{2}(E,\mathbb{Z})\rightarrow H_{2}(Y,\mathbb{Z})\rightarrow H^{2}(Y,\mathbb{Z})\rightarrow H^{2}(E,\mathbb{Z})$
where the first and the last map are induced by the inclusion $E\subset Y$
and the second by Poincar\'e duality. In other words, the intersection
pairing on $Y$ induces a nondegenerate pairing $H_{2}(E,\mathbb{Z})\times H_{2}(E,\mathbb{Z})\rightarrow\mathbb{Z}$,
hence we may view $H_{2}(E,\mathbb{Z})$ as a subgroup of $H^{2}(E,\mathbb{Z})$.
Note that the intersection numbers between irreducible components
of $E$ only depend on the singularities of $X$, so the quotient
$H^{2}(E,\mathbb{Z})/H_{2}(E,\mathbb{Z})$ should be considered as
a local invariant of the singularities of $X$.

\begin{proof}
The existence of the exact sequence follows from \cite[Lemma 2]{Zhang}.
If $\mathrm{Cl}(X)\cong\mathbb{Z}\cdot[-K_{X}]$, then from what we
just said $\mathrm{Pic}(X)\cong\mathbb{Z}\cdot[-rK_{X}]$ hence $\mathrm{Cl}(X)/\mathrm{Pic}(X)\cong\mathbb{Z}/r\mathbb{Z}.$
It remains to prove $H_{1}(X^{0},\mathbb{Z})=0$. Since the intersection
matrix of $E$ is nondegenerate, $H_{1}(X^{0},\mathbb{Z})$ is finite.
If it is not zero, $X^{0}$ will admit a nontrivial \'etale cyclic
covering of degree $d>1$, hence $\mathrm{Pic}(X^{0})\cong\mathrm{Cl}(X)$
would contain $d$-torsion, a contradiction.
\end{proof}

If $p\in X$ is a singular point, we let $r_{p}$ be the local index
of $p$, i.e., the smallest positive integer $m$ such that $mK_{X}$
is Cartier at $p$, and define $\mathrm{Cl}_{p}=H^{2}(E_{p},\mathbb{Z})/H_{2}(E_{p},\mathbb{Z})$
in the same way as in the above lemma with $E_{p}=\pi^{-1}(p)_{\mathrm{red}}$.
As explained in the next lemma, it can be viewed as the ``local class
group'' of $X$ at $p$. Since $(X,p)$ has quotient singularities,
locally (in the analytic topology) it is isomorphic to a neighbourhood in $\mathbb{C}^{2}/G$ of the image of the origin where $G$ is a finite subgroup of $GL(2,\mathbb{C})$,
then $r_{p}=|H|$ where $H$ is the image of $G$ under the determinant
map $\det:G\subset GL(2,\mathbb{C})\rightarrow\mathbb{C}^{*}$, and
$\mathrm{Cl}_{p}$ is isomorphic to the abelianization of $G$:

\begin{lem}
In the above notations, $\mathrm{Cl}_{p}\cong G/G'$.
\end{lem}

\begin{proof}
By definition, $\mathrm{Cl}_{p}$ only depends on the intersection
matrix of $E_{p}$, hence we may replace $X$ by an \'etale neighbourhood
of $p$, in particular we may assume $(X,p)\cong(\mathbb{C}^{2}/G,0)$.
As before $\pi:Y\rightarrow X$ is the minimal resolution, then $E_{p}$
is a deformation retract of $Y$. As $X$ is affine and has rational
singularities, $H^{i}(Y,\mathcal{O}_{Y})=H^{i}(X,\mathcal{O}_{X})=0$
for all $i>0$, so by the long exact sequence associated to the exponential
sequence $0\rightarrow\mathbb{Z}\rightarrow\mathcal{O}_{Y}\rightarrow\mathcal{O}_{Y}^{*}\rightarrow0$
we have $\mathrm{Pic}(Y)\cong H^{2}(Y,\mathbb{Z})\cong H^{2}(E_{p},\mathbb{Z})$
and hence the following commutative diagram (where $U=X\backslash p=Y\backslash E_{p}$
and $E_{p,i}$ are the irreducible components of $E_{p}$):
\[
\xymatrix{\oplus\mathbb{Z}[E_{p,i}]\ar[r]\ar[d]^{\cong} & \mathrm{Pic}(Y)\ar[r]\ar[d]^{\cong} & \mathrm{Pic}(U)\ar[r]\ar[d] & 0\\
H_{2}(E_{p},\mathbb{Z})\ar[r] & H^{2}(E_{p},\mathbb{Z})\ar[r] & \mathrm{Pic}(U)\ar[r] & 0
}
\]
It follows that $\mathrm{Cl}_{p}\cong\mathrm{Pic}(U)$. Let $V=\mathbb{C}^{2}\backslash0$,
then $\mathrm{Pic}(V)=0$ and giving a line bundle on $U$ is equivalent
to giving a $G$-action on the trivial line bundle on $V$ that is
compatible with the $G$-action on $V$. Such objects are classified
by $H^{1}(G,\mathcal{O}_{V}^{*})=H^{1}(G,\mathbb{C}^{*})\cong G/G'$,
so the lemma follows.
\end{proof}

In particular, $r_{p}\le|\mathrm{Cl}_{p}|$. Since $r$ is the lowest
common multiple of all $r_{p}$ and
\[H^{2}(E,\mathbb{Z})/H_{2}(E,\mathbb{Z})\cong\mathbb{Z}/r\mathbb{Z}\]
is the direct sum of all $\mathrm{Cl}_{p}$, we obtain

\begin{cor}
$\mathrm{Cl}_{p}\cong\mathbb{Z}/r_{p}\mathbb{Z}$ for all $p\in\mathrm{Sing}(X)$.
\end{cor}

Quotient surface singularities are classified in \cite[Satz 2.11]{Brieskorn},
using the table there together with the well known classification
of Du Val singularities (see for example \cite{DuVal}) we see that
each singularity of $X$ has to be one of the following: the cyclic
singularity $\frac{1}{n}(1,q)$ where $(q,n)=(q+1,n)=1$, type $\left\langle b;2,1;3,1;3,2\right\rangle $(recall
from \cite[Satz 2.11]{Brieskorn} that a type $\left\langle b;n_{1},q_{1};n_{2},q_{2};n_{3},q_{3}\right\rangle $
singularity is the one whose dual graph is a fork such that the central
vertex represents a curve with self intersection number $-b$ and
the three branches are dual graph of the cyclic singularity $\frac{1}{n_{i}}(1,q_{i})$
$(i=1,2,3)$), or $\left\langle b;2;3;5\right\rangle $ (meaning it
is of type $\left\langle b;2,r;3,s;5,t\right\rangle $ for some $r$,
$s$, $t$) . In particular, $E_{8}$ is the only Du Val singularity
that appears in the list.

We now turn to the classification of surfaces without smooth rational
curves. 

\begin{lem}
\label{lem:only one klt}$X$ has at most one non Du Val singular
point.
\end{lem}

\begin{proof}
Since $X$ satisfies (2) of Proposition \ref{prop:criterion for no P1},
there is an effective divisor $D\in|-K_{X}|$ (which is necessarily
an integral curve). Let $\tilde{D}$ be its strict transform on $Y$,
we may write
\begin{equation}
K_{Y}+\tilde{D}+\sum a_{i}E_{i}=\pi^{*}(K_{X}+D)\sim0\label{eq:discrep}
\end{equation}
where the $E_{i}$'s are the irreducible components of $E$ and $a_{i}\in\mathbb{Z}$
(as $K_{X}+D$ is Cartier on $X$). Since $Y$ is the minimal resolution,
$K_{Y}+\tilde{D}$ is $\pi$-nef, thus by the negativity lemma \cite[Lemma 3.39]{Kollar-Mori},
all $a_{i}\ge0$ and we have $a_{i}\ge1$ if $D$ passes $p=\pi(E_{i})$
or $X$ is not Du Val at $p$. In the latter case, as $K_{X}$ is
not Cartier at $p$, $D$ must pass through $p$.

We claim that $D$ contains at most one singular point of $X$, hence
at most one singular point of $X$ is not Du Val. Suppose this is
not the case, and $p_{1}$, $p_{2}\in D\cap\mathrm{Sing}(X)$, let
$\Delta_{j}=\sum_{\pi(E_{i})=p_{j}}a_{i}E_{i}$ ($j=1,2$), then we
have $\Delta_{j}>0$ and $(\tilde{D}.\Delta_{j})\ge1$. On the other
hand by \eqref{eq:discrep} we have $2p_{a}(\tilde{D})-2+(\Delta_{1}+\Delta_{2}.\tilde{D})=(K_{Y}+\tilde{D}+\Delta_{1}+\Delta_{2}.\tilde{D})\le0$,
hence $p_{a}(\tilde{D})=0$, $\tilde{D}\cong\mathbb{P}^{1}$, and
$(\tilde{D}.\Delta_{j})=1\,(j=1,2)$. As $K_{Y}+\tilde{D}+\Delta_{j}\equiv_{\pi}0$
over $p_{j}$, we can apply \cite[Proposition 5.58]{Kollar-Mori}
to see that $(Y,\tilde{D}+\Delta_{j})$ is lc (hence every curve in
$\Delta_{j}$ appear with coefficient one) and the dual graph of $\tilde{D}+\Delta_{j}$
is a loop, this contradicts the fact that $(\tilde{D}.\Delta_{j})=1$.
\end{proof}

If $X$ is Gorenstein, then by the previous discussion it has only
$E_{8}$-singularities, hence by the classification of Gorenstein log
del Pezzo surfaces, $X$ is one of the two types of $S(E_{8})$ as
discussed in \cite[Lemma 3.6]{Keel-McKernan} and it is straightforward
to verify that neither of them contain smooth rational curves (e.g.
using Proposition \ref{prop:criterion for no P1}). So from now on
we assume $X$ is not Gorenstein, and by the above lemma, we may denote
by $p$ the unique non Du Val singular point of $X$ and let $\Delta=\pi^{-1}(p)_{\mathrm{red}}$
. We also get the following immediate corollary from the proof of
Lemma \ref{lem:only one klt}:

\begin{cor}
\label{cor:D away from Sing(X)}Notation as above, then every effective
divisor $D\sim-K_{X}$ passes through $p$ and no other
singular points of $X$.
\end{cor}

In some cases, the curve $D$ constructed in the previous proof turns
out to be already a smooth rational curve on $X$. To be precise:

\begin{prop}
\label{prop:3 types of Sing(X)}Let $D\in|-K_{X}|$ and $\tilde{D}$
its strict transform on $Y$. Then either $X$ has a cyclic singularity
at $p$ and $K_{Y}+\tilde{D}+\Delta\sim0$, or $(X,p)$ is a singular
point of type $\left\langle b;2,1;3,1;3,2\right\rangle $ or $\left\langle b;2;3;5\right\rangle $
with $b=2$.
\end{prop}

\begin{proof}
We have $K_{Y}+\tilde{D}+\sum a_{i}E_{i}\sim0$ as in \eqref{eq:discrep},
where $a_{i}\in\mathbb{Z}_{>0}$ and $E_{i}\subset\mathrm{Supp}\,\Delta$
by Corollary \ref{cor:D away from Sing(X)}. If some $a_{i}\ge2$,
then $|K_{Y}+\tilde{D}+\Delta|=\emptyset$, hence by Lemma \ref{lem:K_X+D}
$\tilde{D}+\Delta$ is a rational tree, in particular $(\tilde{D}.\Delta)=1$,
and if in addition $\Delta$ is the fundamental cycle of $(X,p)$ (i.e. $-\Delta$
is $\pi$-nef; this is the case if $(X,p)$ has cyclic singularity
or if the central curve of $\Delta$ has self-intersection at most
$-3$), then by \cite[Lemma 4.12]{Keel-McKernan}, $D$ is a smooth
rational curve on $X$, but by Corollary \ref{cor:when a curve is P1},
this contradicts our assumption as $|K_{X}+D|\neq\emptyset$. We already
know that the singularity of $X$ at $p$ is cyclic, $\left\langle b;2,1;3,1;3,2\right\rangle $
or $\left\langle b;2;3;5\right\rangle $, hence in the first case
all $a_{i}=1$, and we claim that in the latter two cases at least
one $a_{i}\ge2$, it would then follow that $b=2$. Suppose all $a_{i}=1$,
then $K_{Y}+\tilde{D}+\Delta\sim0$, but the LHS has positive intersection
with the central curve of $\Delta$, a contradiction.
\end{proof}

We need a more careful analysis in the cyclic case, so assume for
the moment that $X$ has cyclic singularity at $p$. As above, $D$
is an effective divisor in $|-K_{X}|$ and $\tilde{D}$ its birational
transform on $Y$, while $\Delta=\pi^{-1}(p)_{\mathrm{red}}$.

\begin{lem}
\label{lem:D is (-1)-curve}$\tilde{D}$ is a $(-1)$-curve on $Y$.
\end{lem}

\begin{proof}
As $|K_{Y}+\tilde{D}|=|-\Delta|=\emptyset$, $\tilde{D}$ is a smooth
rational curve by Lemma \ref{lem:K_X+D}. We first show that $(\tilde{D}^{2})<0$.
Suppose $(\tilde{D}^{2})\ge0$, then the exact sequence $0\rightarrow\mathcal{O}_{Y}\rightarrow\mathcal{O}_{Y}(\tilde{D})\rightarrow\mathcal{O}_{\tilde{D}}(\tilde{D})\rightarrow0$
and $H^{1}(Y,\mathcal{O}_{Y})=0$ imply that $\mathcal{O}_{Y}(\tilde{D})$
is base point free, hence we can choose $D$ to pass through any point
on $X$. By \ref{cor:D away from Sing(X)}, this implies that $p$
is the unique singular point of $X$. If $C$ is a $(-1)$-curve on
$Y$, then $C$ is not contained in the support of $\tilde{D}+\Delta$,
and as $K_{Y}+\tilde{D}+\Delta\sim0$ we get $(K_{Y}+\tilde{D}+\Delta.C)=-1+(\tilde{D}+\Delta.C)=0$,
hence $(C.\Delta)\le1$ and $\pi(C)$ is a smooth rational curve on
$X$ (using \cite[Lemma 4.12]{Keel-McKernan} as in the proof of \ref{prop:3 types of Sing(X)}),
a contradiction. It follows that $Y$ does not contain any $(-1)$-curves,
hence $Y\cong\mathbb{F}_{e}(e\ge2)$ and $X\cong\mathbb{P}(1,1,e)$,
but then $X$ contains many smooth rational curves, so these cases
won't occur. Hence $(\tilde{D}^{2})<0$. If $(\tilde{D}^{2})\le-2$
then by \cite[Lemma 1.3]{Zhang-klt} $\tilde{D}$ is contained
in $E$ (the exceptional locus of $\pi$), so the lemma follows.
\end{proof}

\begin{lem}
\label{lem:description of cyclic case}There exists a birational morphism
$f:Y\rightarrow\bar{Y}$ such that
\begin{enumerate}
\item $\bar{Y}$ is an $S(E_8)$;
\item $\mathrm{Ex}(f)$ consists of all but one component of $\tilde{D}+E$;
\item $f(\tilde{D})$ is a smooth point on $\bar{Y}$.
\end{enumerate}
\end{lem}

\begin{proof}
Let $Y\rightarrow Y_{0}$ be the contraction of all curves in $E\backslash\Delta$,
then every closed point of $Y_{0}$ is either smooth or an $E_{8}$-singularity
(every Du Val singularity of $X$ is an $E_{8}$-singularity). We
run the $K$-negative MMP starting with $Y_{0}$:
\[
Y_{0}\stackrel{\phi_{1}}{\longrightarrow}Y_{1}\stackrel{\phi_{2}}{\longrightarrow}\cdots\stackrel{\phi_{m}}{\longrightarrow}Y_{m}\stackrel{g}{\longrightarrow}Z
\]
where each step is the contraction of an extremal ray, $\phi_{i}$'s
are birational, and $\dim Z<2$ ($Y_{0}$ is a Gorenstein rational
surface, so the MMP stops at a Mori fiber space). By \cite[Lemma 3.3]{Keel-McKernan},
each $\phi_{i}$ is the contraction of a $(-1)$-curve contained in
the smooth locus of $Y_{i-1}$. If this $(-1)$-curve is not a component of the image of $\tilde{D}+E$, let $C$ be its strict transform in
$Y$, then $C$ is a smooth rational curve with negative self-intersection,
hence by \cite[Lemma 1.3]{Zhang-klt} it is a $(-1)$-curve.
Now the same argument as in Lemma \ref{lem:D is (-1)-curve} shows
that $(C.\Delta)\le1$ and $\pi(C)$ is a smooth rational curve in
$X$, a contradiction. So the exceptional locus of $Y_{0}\rightarrow Y_{m}$
is contained in $\tilde{D}+E$. In particular, since $\tilde{D}$ is the only component of $\tilde{D}+E$ that is a $(-1)$-curve, 
$\phi_{1}$ is the contraction of $\tilde{D}$.

We claim that $Z$ is a point. Suppose it is not, then $g$ is a $\mathbb{P}^{1}$-fibration.
By \cite[Lemma 3.4]{Keel-McKernan}, as $Y_{m}$ has only singularities
of $E_{8}$ type, it is actually smooth and isomorphic to $\mathbb{F}_{e}$
for some $e\ge0$. If $e=1$, one can choose to contract the $(-1)$-curve from $Y_m$ and then $Y_{m+1}=\bP^2$ while $Z$ is a point. So we may assume $e=0$ or $e\ge2$. Since $\mathrm{Cl}(X)$ is generated by $-K_{X}$,
we see that $\mathrm{Cl}(Y)$ is freely generated by $-K_{Y}$ and
the components of $E$, or equivalently, by the components of $\tilde{D}+E$.
Let $\Gamma$ be the image of $\tilde{D}+E$ on $Y_{m}$, we have
$K_{Y_{m}}+\Gamma\sim0$ and the irreducible components of $\Gamma$
freely generate $\mathrm{Cl}(Y_{m})$. As $\rho(Y_{m})=2$
in this case, $\Gamma$ has exactly two irreducible components. However,
this contradicts the next lemma.

Hence $Z$ is a point and $Y_{m}$ is a Gorenstein rank one del Pezzo.
By construction $\mathrm{Cl}(Y_{m})$ is generated by the effective
divisor $\Gamma\sim-K_{Y_{m}}$, in other words, $Y_{m}$ does not
contain any smooth rational curves, hence by the discussion on Du Val case,
$Y_{m}$ is an $S(E_{8})$, and the lemma follows by taking $\bar{Y}=Y_{m}$.
\end{proof}

The following lemma is used in the above proof.

\begin{lem}
Let $S=\mathbb{F}_{e}$ where $e=0$ or $e\ge2$, then $-K_{S}$ cannot be written as
the sum of two irreducible effective divisors that generate $\mathrm{Pic}(S)$.
\end{lem}

\begin{proof}
It is quite easy to see that when $e=0$ such a decomposition of $-K_S$ is not possible, so we assume $e\ge2$.
Let $C_{0}$ be the unique section of negative self-intersection and
$F$ be a fiber, then $\mathrm{Pic}(S)$ is freely generated by $C_{0}$
and $F$. If $M=aC_{0}+bF$ represents an irreducible curve then $M=aC_{0}$,
or $b\ge ae\ge0$. Suppose $-K_{S}\sim2C_{0}+(e+2)F\sim M_{1}+M_{2}$
where $M_{1}$ and $M_{2}$ are irreducible and generate $\mathrm{Pic}(S)$.
Then we must have $M_{i}=C_{0}+m_{i}F$ with $m_{i}\ge e$ and $m_{1}+m_{2}=e+2$,
this is only possible when $m_{1}=m_{2}=e=2$, but then $M_{1}=M_{2}$
can not generate $\mathrm{Pic}(S)$.
\end{proof}

Back to the general case. To finish the classification, let us now
construct some surfaces that satisfy the conditions in Proposition
\ref{prop:criterion for no P1}. Let $\bar{Y}$ be an $S(E_{8})$
with $\Gamma\in|-K_{\bar{Y}}|$ a rational curve. We have $\Gamma\subset\bar{Y}^{0}$
and $(\Gamma^{2})=1$. Let $q$ be the unique double point of $\Gamma$.
Let $Y\rightarrow\bar{Y}$ be the blowup at $q_{1}=q,$ $q_{2}$,
$\cdots$ , $q_{m}$ where each $q_{i}$ is infinitely near $q_{i-1}$
($i>1$). If $q$ is a node of $\Gamma$, we also require that $q_{i}$
always lies on the strict transform of either $\Gamma$ or exceptional
curves of previous blowup (there are 2 different choices of $q_{i}$
for each $i>1$). If $q$ is a cusp then we require that $m=1,2$ or $4$
and that $q_{i}$ lies on the strict transform of $\Gamma$ for $i=2,3$ while $q_{4}$
is away from $\Gamma$ and previous exceptional curves. Let $E_i$ be the strict transform of the exceptional curve coming from the blowup of $q_i$. We define
$X(\bar{Y},\Gamma;q_{1},\cdots,q_{m})$ to be the contraction from
$Y$ of $\Gamma$ and $E_i$ ($i=1,\cdots,m-1$). It has two singular points, one of which
is an $E_{8}$ singularity and the other is a cyclic singularity except
when $\Gamma$ has a cusp at $q$ and $m=4$, in which case the second
singularity has type $\left\langle 2;2,1;3,1;5,1\right\rangle$. Argue inductively, we get $-K_Y\sim \Gamma+\sum_{i=1}^m E_i$ unless $\Gamma$ has a cusp at $q$ and $m=4$, in which case we have $-K_Y\sim \Gamma+E_1+E_2+2E_3+E_4$ instead. As $\mathrm{Cl}(Y)$ is generated by $\Gamma$ and all the $E_i$, it is not hard to verify that $X(\bar{Y},\Gamma;q_{1},\cdots,q_{m})$
satisfies condition (2) in Proposition \ref{prop:criterion for no P1}.

\begin{thm}
\label{thm:Classification}If $X$ is a surface with only
quotient singularities that satisfies the conditions in Proposition
\ref{prop:criterion for no P1}, then it is either an $S(E_{8})$
or one of the $X(\bar{Y},\Gamma;q_{1},\cdots,q_{m})$ constructed
above.
\end{thm}

\begin{proof}
If $X$ is Gorenstein then it is an $S(E_{8})$, so we may assume
that $X$ is not Gorenstein. Let $p$ be its unique non Du Val singular
point, by Proposition \ref{prop:3 types of Sing(X)}, there are 3
possibilities for the singularity of $(X,p)$, and we analyse them
one by one:
\begin{enumerate}
\item $(X,p)$ is a cyclic singularity. Let $Y_0$ be as in Lemma \ref{lem:description of cyclic case}. By Lemma \ref{lem:description of cyclic case}, there exists a birational morphism $f:Y_0\rightarrow \bar{Y}$ where $\bar{Y}$ is an $S(E_8)$ such that $f$ contracts all but one component of $\tilde{D}+\Delta$ to a smooth point (we use the same letters for strict transforms of $\tilde{D}$ and $\Delta$ on $Y_0$). By Proposition \ref{prop:3 types of Sing(X)}, $K_{Y_0}+\tilde{D}+\Delta\sim0$, thus the dual graph of $\tilde{D}+\Delta$ is a loop. It follows that $\Gamma=f(\tilde{D}+\Delta)\sim -K_{\bar{Y}}$ is a rational curve with a double point $q$. In addition, $q$ is a cusp if and only if $\tilde{D}+\Delta$ consists of two rational curves that are tangent to each other or three rational curves that intersect at the same point. In particular, $\tilde{D}+\Delta$ has at most three components when $q$ is a cusp. Since $f$ is a composition of blowing down of $(-1)$-curves, we recover $Y_0$ as a successive blowup from $\bar{Y}$ of nodes on the images of $\tilde{D}+\Delta$. Let $q_1,\cdots,q_m$ be the centers of these blowups. Clearly $q_1=q$ and if $q$ is a cusp then $m\le2$. As $\tilde{D}$ is the only $(-1)$-curve among the components of $\tilde{D}+\Delta$, each $q_i$ is infinitely near $q_{i-1}$. It is now easy to see that $X$ is a $X(\bar{Y},\Gamma;q_{1},\cdots,q_{m})$ with $\Gamma$ nodal or $\Gamma$ cuspidal and $m\le2$.

\item $(X,p)$ has type $\left\langle 2;2,1;3,1;3,2\right\rangle $. By
assumption $H^{2}(Y,\mathrm{Z})=\mathrm{Pic}(Y)$ is freely generated by
$K_{Y}$ and the components in $E$. Since the intersection paring
on $H^{2}(Y,\mathbb{Z})$ is unimodular, the intersection matrix of
$K_{Y}$ and $E$ has determinant $\pm1$. Write $K_{Y}+G=\pi^{*}K_{X}$
where $G$ is supported on $E$ (and can be easily computed from the
given singularity type). As $\pi^*K_X$ is the orthogonal projection of $K_Y$ to the span of the components of $E$, we must then have $(K_X^2)=\left((K_{Y}+G)^{2}\right)=(K_{Y}^{2})+(K_{Y}.G)=10-\rho+(K_{Y}.G)=\frac{1}{r}$
where $r=|\det\left((E_{i}.E_{j})\right)|$ and $\rho$ is the Picard
number of $Y$. It is straightforward to compute that $G=\frac{5}{9}E_1+\cdots$ where $E_1$ is only component of $E$ with self-intersection $(-3)$ (and this is the only component whose coefficient is relevant to us), $(K_{Y}.G)=\frac{5}{9}$
and $r=9$. But $\rho$ is an integer, so this case cannot occur.

\item $(X,p)$ has type $\left\langle 2;2;3;5\right\rangle $. A similar
computation as in case (2) shows that in order to have $10-\rho+(K_{Y}.G)=\frac{1}{r}$, $(X,p)$ must has type $\left\langle 2;2,1;3,1;5,1\right\rangle $ and $\rho=13$. Since the other singularities of $X$ are of $E_8$-type, $X$ has exactly one $E_{8}$-singularity. By the same proof as the proof of Lemma
\ref{lem:D is (-1)-curve}, $\tilde{D}$ is a $(-1)$-curve. Let $E_1$ be the central curve of $\Delta$ and $E_2$, $E_3$, $E_5$ the other three components of $\Delta$ with self-intersections $-2$, $-3$ and $-5$ respectively. Write $K_Y+\tilde{D}+\sum a_iE_i=\pi^*(K_X+D)\sim0$
as before. We have $a_1\ge2$ since otherwise the LHS has positive intersection with $E_1$. By Lemma \ref{lem:K_X+D}, $\tilde{D}+\Delta$ is a rational tree, thus $D$ intersects transversally with exactly one component of $\Delta$. It is straightforward to find the discrepancies $a_i$ once we know which component $\tilde{D}$ intersects. But as $a_i$'s are integers, we find that $\tilde{D}$ intersects $E_1$ by enumerating all the possibilities and that $a_2=a_3=a_5=1$. Now as in Lemma \ref{lem:description of cyclic case} we may contract $\tilde{D}$, $E_1$, $E_2$, $E_3$ and all components of $E\backslash\Delta$ from $Y$ to obtain $\bar{Y}$, which is an $S(E_8)$, such that the image of $E_5$ is a cuspidal rational curve $\Gamma\sim -K_{\bar{Y}}$. Reversing this blowing down procedure we see that $X$ is isomorphic
to some $X(\bar{Y},\Gamma;q_{1},\cdots,q_{4})$ where $\bar{Y}$ is
an $S(E_{8})$ and $\Gamma$ is cuspidal.
\end{enumerate}
\end{proof}

It is well known that $E_{8}$ is the only surface quotient singularity
that does not admit a smooth curve germ \cite{SmoothArc}. Hence the
following corollary follows immediately from the above theorem.

\begin{cor} \label{cor:noE8}
Let $X$ be a surface with only quotient singularities. Assume
that
\begin{enumerate}
\item $K_{X}$ is not pseudo-effective;
\item For every closed point $p\in X$, there is a smooth curve (locally
analytically) passing through $p$.
\end{enumerate}
Then $X$ contains at least one smooth rational curve.
\end{cor}

\section{Examples and questions}

If $X$ is a log del Pezzo surface, a curve of minimal degree on
$X$ seems to be a natural candidate for the smooth rational curve
(such a curve is used extensively in the study of log del Pezzo surfaces).
However, the following example shows that this is not always the case,
even if $X$ is known to contain some smooth rational curve.

\begin{example}
Let $Y\neq S(E_{8})$ be a Gorenstein log del Pezzo surface of degree
1 such that the linear system $|-K_{Y}|$ contains a nodal curve $D$.
Let $\bar{Y}\rightarrow Y$ be the blow up of the node of $D$. Let
$E$ be the exceptional curve and $\bar{D}$ the strict transform
of $D$. Contract the $(-3)$-curve $\bar{D}$ to get our surface
$X$. It is straightforward to verify that the image of $E$ under
the contraction is the only curve of minimal degree on $X$. But since
$E$ intersects $\bar{D}$ at two points, its image on $X$ is not
smooth. In fact the smooth rational curves on $X$ are usually given
by the strict transform of $(-1)$-curves on the minimal resolution
of $Y$. Observe that as $K_{X}+E\sim0$, we have that $K_{X}+C$ is ample
for any smooth rational curve $C$ on $X$.
\end{example}

We are also interested in whether the smooth rational curve $C$ we
find supports a tiger of the log del Pezzo surface $X$ (i.e. there exists $D\sim_\bQ -K_X$ such that $\mathrm{Supp}(D)=C$ and $(X,D)$ is not klt. See \cite[Definition 1.13]{Keel-McKernan}). At least when $C$ passes
through at most one singular point we have a positive answer:

\begin{lem}
Let $C$ be a smooth rational curve on a rank 1 log del Pezzo surface
$X$. Assume that $C$ passes through at most one singular point of
$X$. If $\alpha\in\mathbb{Q}$ is chosen such that $K_{X}+\alpha C\equiv0$,
then the pair $(X,\alpha C)$ is not klt.
\end{lem}

\begin{proof}
If $C$ lies in the smooth locus of $X$ then by adjunction $(K_{X}+C.C)=-2<0$,
hence $\alpha>1$ and the result is clear. Otherwise we may assume
$C\cap\mathrm{Sing}(X)=\{p\}$. Let $\beta$ be the log canonical
threshold of the pair $(X,C)$ and $\pi:\tilde{X}\rightarrow X$ the
minimal resolution. It suffices to show that $(K_{X}+\beta C.C)\le0$.
As $C$ is a smooth rational curve, $\pi$ is also a log resolution
of $(X,C)$. Write $\pi^{*}(K_{X}+\beta C)=K_{\tilde{X}}+\beta\tilde{C}+\sum a_{i}E_{i}$
where the $E_{i}$'s are the exceptional curves of $\pi$. We have
$a_{i}\le1$ by the choice of $\beta$ and $\tilde{C}$ only intersects
one $E_{i}$. Now since $X$ is of rank 1 we have $(\tilde{C}^{2})\ge-1$
by \cite[Lemma 1.3]{Zhang-klt} and $(K_{\tilde{X}}+\tilde{C}.\tilde{C})=-2$
by adjunction, thus
\[
(K_{X}+\beta C.C)=(K_{\tilde{X}}+\beta\tilde{C}+\sum a_{i}E_{i}.C)\le-1-\beta+\sum(E_{i}.\tilde{C})\le-\beta<0
\]
\end{proof}

On the other hand, once $C$ passes through more singular points of
$X$, the situation becomes more complicated. The following example
suggests that even if $X$ has a tiger, there is in general no guarantee
that the tiger can be supported on $C$.

\begin{example}
Similar to the previous example, let $Y$ be a Gorenstein log del
Pezzo surface of degree 1 with an $A_{8}$-singularity and $D\in|-K_{Y}|$
a nodal curve. Blow up the node and one of its infinitely near points
to get a new surface $\bar{Y}$ and let $X$ be the contraction of
the strict transform of $D$ and the first exceptional curve. The
second singularity of $X$ has dual graph of type $\left\langle 4,2\right\rangle $.
Every smooth rational curve on $X$ is a $(-1)$-curve on the minimal
resolution and intersects both singular points of $X$. By direct
computation we have $\beta=\frac{1}{2}$ (where $\beta=\mathrm{lct}(X,C)$
as in the proof of the above lemma) and $(K_{X}+\beta C.C)>0$, hence
by the same reasoning for the above lemma we know that $C$ does not
support a tiger. However, $-K_{X}$ is effective so $X$ does have
a tiger.
\end{example}

In view of Proposition \ref{prop:criterion for no P1}, we may ask for a similar classification of surfaces with rational singularities that do not contain smooth rational curves. The next example shows that we do get additional cases.

\begin{example}
The construction is similar to that of $X(\bar{Y},\Gamma;q_{1},\cdots,q_{m})$. Let $\bar{Y}$ be an $S(E_{8})$ with $\Gamma\in|-K_{\bar{Y}}|$ a cuspidal rational curve and let $q$ be the cusp of $\Gamma$. Let $Y\rightarrow\bar{Y}$ be the blowup at $q_{1}=q,$ $q_{2}$,
$\cdots$ , $q_{m}$ ($m\ge5$) where each $q_{i}$ is infinitely near $q_{i-1}$
($i>1$) such that $q_{i}$ lies on the strict transform of $\Gamma$ for $i<m$ while $q_m$ is away from $\Gamma$ and the previous exceptional curves. Let $E_i$ be the strict transform of the exceptional curve coming from the blowup of $q_i$. The dual graph of $\Gamma$ and $E_i$ ($i=1,\cdots,m-1$) is given as follows:
\begin{center}
\begin{tikzpicture}
\node (0) [anchor = west] {$(-2)$};
\node (1) [right = 0.25cm of 0, anchor = west] {$(-2)-\cdots-(-2)-(-m-1)$};
\node (2) [above left = 0.5cm of 0, anchor = east] { $(-2)$};
\node (3) [below left = 0.5cm of 0, anchor = east] { $(-3)$};
\path
    (0) edge (1)
    (2) edge (0)
    (3) edge (0);
\end{tikzpicture}
\end{center}
We define $X$ to be the contraction from $Y$ of these curves. It has two singular points, one of which is an $E_{8}$ singularity and the other is not a quotient singularity since we assume $m\ge5$. Nevertheless, it is a rational singularity (one way to see this is to attach $m$ auxiliary $(-1)$-curves to $\Gamma$ and notice that the corresponding configuration of curves contracts to a smooth point, hence any subset of these curves also contracts to a rational singularity by \cite[Proposition 1]{artin}). We also have $-K_Y\sim \Gamma+E_1+E_2+2\sum_{i=3}^{m-1}E_i+E_m$ by induction on $m$ and it follows as before that $\mathrm{Cl}(X)$ is generated by the image of $E_m$ which is linearly equivalent to $-K_X$. By Proposition \ref{prop:criterion for no P1}, $X$ is a surface with rational singularities that doesn't contain any smooth rational curves.
\end{example}

We observe that the surfaces in the above example still contain $E_8$ singularities and thus violate the second assumption of Corollary \ref{cor:noE8}. In addition the above construction does not seem to have many variants. It is therefore natural to ask the following question:

\begin{que}
Let $X$ be a surface with rational singularities. Assume that $K_X$ is not pseudo-effective and every closed point of $X$ admits a smooth curve germ. Is it true that $X$ contains a smooth rational curve? More aggressively, classify all surfaces with rational singularities that do not contain smooth rational curves.
\end{que}

Finally we investigate what happens if we remove the assumption on
$K_{X}$ in our main theorem. Clearly there are many smooth surfaces (e.g. abelian surfaces, ball quotients, etc.) with nef canonical divisors that do not even contain rational curves. Since we are mostly interested in the existence of \emph{smooth} rational curves, we restrict ourselves to rational surfaces. We will construct some examples of rational surfaces with cyclic quotient singularities that do not contain smooth rational curves. These rational surfaces $X$ will be the quotient of certain singular K3 surfaces and satisfy $K_{X}\sim_\bQ 0$, hence the
assumption (1) in our main theorem is necessary.

\begin{example} \label{ex:no P1}
Let $T$ be a smooth del
Pezzo surface of degree 1. For general choice of $T$, the linear
system $|-K_{T}|$ contains at least two nodal rational curves $C_{i}$
($i=1,2$). Let $Q_{i}$ be the node of $C_{i}$ and $P=C_{1}\cap C_{2}$.
Let $\pi:Y\rightarrow T$ be the blowup of both $Q_{i}$ with exceptional divisors $E_i$ and let $\tilde{C}_{i}$ be the strict transform of $C_{i}$ on $Y$. Then $K_Y=\pi^*K_T+E_1+E_2$ and $\tilde{C}_i=\pi^*C_i-2E_i=\pi^*(-K_T)-2E_i$, thus $-2K_Y\sim \tilde{C}_1+\tilde{C}_2$. We also have $(\tilde{C}_{i}^{2})=-3$ and $(\tilde{C}_{1}.\tilde{C}_{2})=1$, hence we can contract both
$\tilde{C}_{i}$ simutaneously to get a rational surface $X$ with
a cyclic singularity $p$ of type $\frac{1}{8}(1,3)$. The next three lemmas tell us that
for very general choice of $T$ and $C_{i}$, such $X$ does not contain
any smooth rational curves.
\end{example}

\begin{lem}
\label{lem:example_curve at p}Every smooth curve on $X$ is away
from $p$.
\end{lem}

\begin{proof}
Let $p\in C$ be a smooth curve on $X$ and $\tilde{C}$ its strict
transform on $Y$, then $(\tilde{C}.\tilde{C}_{1}+\tilde{C}_{2})=1$.
But we have $\tilde{C}_{1}+\tilde{C}_{2}=-2K_{Y}$, so the intersection
must be even, a contradiction.
\end{proof}

Let $Y'\rightarrow Y$ be the blowup of $P$ and $C_{i}'$ the strict
transform of $\tilde{C}_{i}$, then $C_{1}'+C_{2}'=-2K_{Y'}$, hence
we can take the double cover $f:S\rightarrow Y'$ ramified along $C_{1}'+C_{2}'$. 
$S$ is smooth as $C_{1}'$ and $C_{2}'$ are smooth and disjoint.
$S$ is indeed a K3 surface as $K_{S}=f^{*}K_{Y'}(\frac{1}{2}(C_{1}'+C_{2}'))\sim 0$
and $H^{1}(S,\mathcal{O}_{S})=H^{1}(Y',\mathcal{O}_{Y'})\oplus H^{1}(Y',\mathcal{O}_{Y'}(K_{Y'}))=0$.

\begin{lem}
\label{lem:example_picard number}For very general choice of $T$ and $C_{i}$,
the above K3 surface $S$ has Picard number 12.
\end{lem}

\begin{proof}
$\rho(S)\ge12$ as it's a double cover of $Y'$ and $\rho(Y')=12$.
Since the moduli space of K3 surfaces is 20-dimensional, the locus
of those with Picard number at least 13 is a countable union of subvarieties
of dimension at most 7. On the other hand, the above construction
gives us an 8-dimensional family of K3 surfaces: we have an 8-dimensional family
of del Pezzo surfaces of degree 1. Hence for very general choice of $T$,
we get a K3 surface $S$ with $\rho(S)=12$.
\end{proof}

It now follows that

\begin{lem}
For very general choice of $C_{i}$, the rational surface $X$ constructed
above does not contain any smooth rational curve.
\end{lem}

\begin{proof}
Suppose $C\subseteq X$ is a smooth rational curve. By Lemma \ref{lem:example_curve at p},
$p\not\in C$, hence its strict transform $C'$ in $Y'$ is disjoint
from $C_{1}'+C_{2}'$. As $f:S\rightarrow Y'$ is \'etale outside
$C_{1}'+C_{2}'$, $f^{-1}(C')$ splits into a disjoint union of two
smooth rational curves $D_{1}$, $D_{2}$. This implies $\rho(S)\ge13$
($D_{1}$, $D_{2}$ and the pullback of the orthogonal complement
of $C'$ in $\mathrm{Pic}(Y')$ generate a sublattice of rank 13),
which can not happen for very general choice of $T$ and $C_{i}$ by
Lemma \ref{lem:example_picard number}.
\end{proof}

By allowing more singular points, we can give a similar construction
with a simpler proof of non-existence of smooth rational curves.

\begin{example}
Instead of taking a smooth del Pezzo surface of degree 1, let $T$
be a Gorenstein rank one log del Pezzo surface of degree 1. Assume
either $T$ has a unique singular point or it has exactly two $A_{n}$-type
singular points, then a similar argument as the proof of \cite[Lemma 3.6]{Keel-McKernan} imlies
that for general choice of $T$, $|-K_{T}|$ contains two nodal rational
curves $C_{i}\,(i=1,2)$ lying inside the smooth locus of $T$. Let
$X$ be the surface obtained by the same construction in Example \ref{ex:no P1}
(i.e. blow up the nodes $Q_{i}$ of $C_{i}$ and contract both $\tilde{C}_{i}$),
then it has the same singularities as $T$ as well as a cyclic singularity
$p$ of type $\frac{1}{8}(1,3)$. Suppose $X$ contains a smooth rational
curve $C$. As before we know that $C\subset U=X\backslash p$, and
as $2K_{X}\sim0$, we have a double cover $g:Y\rightarrow X$ that
is unramified over $U$ (since $K_{X}$ is Cartier over $U$). Since
$C\cong\mathbb{P}^{1}$ is simply connected, we see that $g^{-1}(C)$
consists of two disjoint copies of $\mathbb{P}^{1}$. By construction
$X$ has Picard number one, hence $C$ is ample, thus $g^{*}C$ is
also ample on $Y$, but this contradicts \cite[III.7.9]{hartshorne}. In some cases one can
also derive a contradiction without using the double cover. For example
suppose $T$ has a unique $A_{8}$-type singularity $q$, then modulo
torsion $\mathrm{Cl}(X)$ is generated by $E$, the strict transform
of the exceptional curve over either one of the $Q_{i}$'s, and $(E^{2})=\frac{1}{2}$.
It follows that
\begin{equation}
\deg(K_{C}+\mathrm{Diff}_{C}(0))=(K_{X}+C.C)=(C^{2})\ge\frac{1}{2}\label{eq:deg}
\end{equation}
but $\deg K_{C}=-2$ and as $C$ is smooth at $q$, the dual graph
of $(X,C)$ at $q$ is a fork with $C$ being one of the branches. It
is then straightforward to compute that $\deg\mathrm{Diff}_{C}(0)=(\frac{1}{m}+\frac{1}{n})^{-1}\le\frac{20}{9}$
where $m,n$ are the index of the other two branches of the dual graph
(i.e. one larger than the number of vertices in the branch), which
contradicts (\ref{eq:deg}).
\end{example}

Inspired by these examples, we may expect to take certain quotients of Calabi-Yau varieties and construct higher-dimension rationally connected varieties with klt singularities that do not contain smooth rational curves. Unfortunately we are unable to identify such an example, and therefore leave it as a question:

\begin{que}
Let $X$ be a rationally connected variety of dimension $\ge3$ with klt singularities. Does $X$ always contain a smooth rational curve? 
\end{que}

We remark that if $X$ is indeed log Fano, then a folklore conjecture predicts that the smooth locus of $X$ is rationally connected and thus contains a smooth rational curve since $\dim X\ge 3$.

\bibliography{ref}
\bibliographystyle{alpha}

\end{document}